\documentclass[11pt, reqno]{amsart}
\usepackage[utf8]{inputenc}
\usepackage{indentfirst, amssymb, amsmath, amsthm, mathrsfs, setspace, indentfirst, enumerate,  mathrsfs, amsmath, amsthm}
\usepackage[colorlinks=true,linkcolor=purple, citecolor=blue,urlcolor=magenta]{hyperref}
\usepackage{graphicx}      
\usepackage{float}         
\usepackage{xcolor}        
\usepackage{colortbl}     
\usepackage{multirow}     
\usepackage{tikz}          
\definecolor{sectionlink}{RGB}{0,100,200} 
\newtheorem{ques}{Question}[section]
\newtheorem{prop}{Proposition}[section]
\newtheorem{theo}{Theorem}[section]
\newtheorem{lem}{Lemma}[section]

\newtheorem{exm}{Example}[section]
\newtheorem{defi}{Definition}[section]
\newtheorem{rem}{Remark}[section]
\newtheorem*{theoA}{Theorem A}
\newtheorem*{theoB}{Theorem B}
\newtheorem*{theoC}{Theorem C}

\numberwithin{equation}{section}
\newcommand{\beas}{\begin{eqnarray*}}
\newcommand{\eeas}{\end{eqnarray*}}
\newcommand{\bea}{\begin{eqnarray}}
\newcommand{\eea}{\end{eqnarray}}

\begin{document}

\title[Sharp Estimates for Hankel, Fekete--Szeg\"o and Zalcman Functionals for $\mathcal{S}_\mathbb{B}^{*}(\alpha)$...]{Sharp Estimates for Hankel, Fekete--Szeg\"o and Zalcman Functionals for $\mathcal{S}_\mathbb{B}^{*}(\alpha)$ in Complex Banach Spaces}
\author[ N. Sarkar and P. Das]{Nabadwip Sarkar and Pradip Das}
\address{Amity School of Applied Sciences, Amity University Mumbai, Panvel, Navi Mumbai, Maharashtra-410206, India}
\email{nsarkar@mum.amity.edu, nabadwipsarkar52@gmail.com}
\address{Department of Mathematics, Raiganj University, Raiganj, West Bengal-733134, India.}
\email{pradipsmath@gmail.com}

\makeatletter
\@namedef{subjclassname@2020}{\textup{2020} Mathematics Subject Classification}
\makeatother

\begin{abstract}
Inspired by the sharp coefficient estimates established by Cho \emph{et al.}\cite{CKKLS2018} for starlike functions of order $\alpha$ in the unit disk, we investigate the corresponding problems for starlike mappings of order $\alpha$ defined on the unit ball of a complex Banach space. Employing Fr\'echet derivatives together with suitable auxiliary lemmas, we establish sharp upper bounds for the second-order Hankel determinant, the Fekete--Szeg\"o functional, and the Zalcman functional associated with this class of mappings. In each case, the obtained estimates are shown to be sharp by identifying the corresponding extremal mappings. Furthermore, our results reduce to the known one-dimensional sharp estimates when the underlying Banach space is the complex plane, thereby extending several classical results of Cho \emph{et al.} \cite{CKKLS2018}  to the setting of complex Banach spaces.
\end{abstract}

\maketitle

\section{{\bf Introduction and Preliminaries.}}

Let $\mathcal{A}$ denote the family of analytic functions in the open unit disk $\mathbb{U}=\{z\in\mathbb{C}:|z|<1\},$
normalized by
\begin{equation}\label{eq:A}
f(z)=z+\sum_{m=2}^{\infty}a_mz^m.
\end{equation}
Furthermore, let $\mathcal{S}\subset\mathcal{A}$ be the class of normalized univalent functions, and let $\mathcal{K}$ denote the subclass of $\mathcal{S}$ consisting of convex functions.

Throughout this paper, we shall frequently use the classical Carath\'eodory class, denoted by $\mathcal{P}$, which consists of analytic functions $p$ satisfying
\[
p(0)=1
\quad\text{and}\quad
\operatorname{Re}p(z)>0,\qquad z\in\mathbb{U}.
\]
Every function $p\in\mathcal{P}$ admits the Taylor expansion
\begin{equation}\label{eq:P}
p(z)=1+\sum_{m=1}^{\infty}c_mz^m
=1+c_1z+c_2z^2+c_3z^3+\cdots,
\qquad z\in\mathbb{U}.
\end{equation}

Let $X$ be a complex Banach space endowed with the norm $\|\cdot\|$, and let
\[
\mathbb{B}=\{x\in X:\|x\|<1\}
\]
be its open unit ball. We denote by $\mathcal{L}(X,Y)$ the Banach space of all bounded linear operators from $X$ into another complex Banach space $Y$. The identity operator on $X$ is denoted by $I$.

For each nonzero vector $x\in X$, define
\[
T(x)=
\left\{
T_x\in\mathcal{L}(X,\mathbb{C}):
T_x(x)=\|x\|,
\ \|T_x\|=1
\right\}.
\]
The Hahn--Banach theorem guarantees that $T(x)$ is nonempty. Moreover, for every $\xi\in\mathbb{C}\setminus\{0\}$,
\[
T_{\xi x}(\cdot)=\frac{|\xi|}{\xi}\,T_x(\cdot),
\]
which establishes a one-to-one correspondence between $T(\xi x)$ and $T(x)$.

Let $\mathcal{H}(\mathbb{B})$ denote the family of all holomorphic mappings from $\mathbb{B}$ into $X$. If $F\in\mathcal{H}(\mathbb{B})$, then for every $x\in \mathbb{B}$, the Fr\'echet--Taylor expansion is given by
\[
F(y)
=
\sum_{n=0}^{\infty}
\frac{1}{n!}
D^nF(x)\bigl((y-x)^n\bigr),
\]
for all $y$ sufficiently close to $x$, where
\[
D^nF(x)\bigl((y-x)^n\bigr)
=
D^nF(x)(y-x,\ldots,y-x),
\]
and $D^nF(x)$ denotes the $n$th Fr\'echet derivative of $F$ at $x$. Each $D^nF(x)$ is a bounded symmetric $n$-linear operator from $X^n$ into $X$.

A mapping $F\in\mathcal{H}(\mathbb{B})$ is called \emph{biholomorphic} if $F(\mathbb{B})$ is a domain in $X$ and the inverse mapping
\[
F^{-1}:F(\mathbb{B})\rightarrow \mathbb{B}
\]
exists and is holomorphic. Moreover, $F$ is said to be \emph{locally biholomorphic} whenever the Fr\'echet derivative $DF(x)$ is invertible with bounded inverse for every $x\in \mathbb{B}$.

A holomorphic mapping $F:\mathbb{B}\rightarrow X$ is said to be \emph{normalized} if
\[
F(0)=0
\quad\text{and}\quad
DF(0)=I.
\]
Several important subclasses of normalized biholomorphic mappings in Banach spaces can be found in the monograph~\cite{GK2003}.

For later use, let $x_0\in X$ satisfy $\|x_0\|=1$ and choose $T_{x_0}\in T(x_0)$. We introduce the quantities $A_1=1$
and, for $n=2,3,4,5$,
\begin{equation}\label{eq:An}
A_n=\frac{1}{n!}T_{x_0}\left(D^nF(0)(x_0,\ldots,x_0)
\right),
\end{equation}
where the vector $x_0$ appears exactly $n$ times in the multilinear form $D^nF(0)$.

The following proposition and defination are the well-known criterion for starlikeness on the unit
ball in a complex Banach space.

\begin{prop}[{\cite[Proposition 1.1]{Suffridge1973}}]
Let $F : \mathbb{B}\to X$ be a locally biholomorphic mapping. Then $F$ is a starlike mapping on $B$ if and only if
\begin{equation}\label{p1}
\operatorname{Re}\left\{T_x\!\left((DF(x))^{-1}F(x)\right)\right\}>0,
\quad x\in \mathbb{B}\setminus\{0\},\; T_x\in T(x).
\end{equation}
\end{prop}

\begin{defi}[{\cite[Definition 1.2]{HKL2001}}]
Let $F : \mathbb{B}\to X$ be a normalized locally biholomorphic mapping and let
$0<\alpha<1$. We say that $F$ is a \emph{starlike mapping of order $\alpha$} if
\[
\left|
\frac{1}{\|x\|}
T_x\!\left((DF(x))^{-1}F(x)\right)
-\frac{1}{2\alpha}
\right|
<
\frac{1}{2\alpha},
\quad
x\in \mathbb{B}\setminus\{0\},\;
T_x\in T(x).
\]

Equivalently,
\begin{equation}\label{d1}
\operatorname{Re}\left\{
\frac{1}{\|x\|}
T_x\!\left((DF(x))^{-1}F(x)\right)
\right\}
>\alpha,
\quad
x\in \mathbb{B}\setminus\{0\},\;
T_x\in T(x).
\end{equation}

Let $\mathcal{S}_\mathbb{B}^{*}(\alpha)$ denote the class of starlike mappings of order $\alpha$ on $B$.
When $X=\mathbb{C}$ and $B=\mathbb{U}$, the class $\mathcal{S}_\mathbb{B}^{*}(\alpha)$ is denoted by
$\mathcal{S}^{*}(\alpha)$.
\end{defi}

\begin{rem}
When $X=\mathbb{C}$ and $B=\mathbb{U}$, the relations \eqref{p1} and \eqref{d1}
reduce to the classical analytic characterizations
\[
\operatorname{Re}\left(\frac{zf'(z)}{f(z)}\right)>0,
\qquad z\in\mathbb{U},
\]
and
\[
\operatorname{Re}\left(\frac{zf'(z)}{f(z)}\right)>\alpha,
\qquad
0<\alpha<1,\; z\in\mathbb{U},
\]
which characterize starlike functions and starlike functions of order $\alpha$ in the unit disk, respectively.
\end{rem}

We first recall the definitions of the Hankel determinants for functions in the class $\mathcal{A}$.

\begin{defi}
Let \(f\in\mathcal{A}\) be given by (\ref{eq:A}). Then, for integers \(q\geq 1\) and \(n\geq 0\), the \(q\)th Hankel determinant is defined by
\[
H_{q,n}(f):=
\begin{vmatrix}
a_n & a_{n+1} & \cdots & a_{n+q-1}\\
a_{n+1} & a_{n+2} & \cdots & a_{n+q}\\
\vdots & \vdots & \ddots & \vdots\\
a_{n+q-1} & a_{n+q} & \cdots & a_{n+2q-2}
\end{vmatrix}.
\]
In particular,
\[
H_{2,2}(f)=a_2a_4-a_3^2 \;\text{and}\;\;H_{2,1}(f)=a_3-a_2^2.
\]
\end{defi}

In recent years, the problem of determining sharp upper bounds for the Hankel determinant
\(|H_{q,n}(f)|\) has attracted considerable attention in geometric function theory. In particular,
the functional $H_{2,1}(f)=a_3-a_2^2,$ coincides with the classical Fekete--Szeg\"o functional introduced by Fekete and Szeg\"o
\cite{FS1933}. For the class $\mathcal{S}$, this functional was first estimated by Bieberbach
(see \cite[Vol.~I, p.~35]{Goodman1983}). Subsequently, Pommerenke \cite{Pommerenke1966} established a fundamental
result for the class $\mathcal{S}$, which stimulated extensive research on analogous coefficient
problems for various subclasses of univalent functions. More recently, considerable effort has
been devoted to obtaining sharp estimates for the second Hankel determinant $H_{2,2}(f)=a_2a_4-a_3^2,$
and several significant results have been reported in the literature (see, for example,
\cite{CKKLS2018,CKKLS2017,XDL2026,XHX2025}).\\

We now turn our attention to the Zalcman functional, its connection with the celebrated Zalcman conjecture and the generalized formulation introduced by Ma \cite{M1999}. \\
In the early 1970s, Lawrence Zalcman conjectured that if \(f\in\mathcal{S}\) is of the form (\ref{eq:A}), then
\[
\left|a_n^2-a_{2n-1}\right|\le (n-1)^2,\qquad n\ge 2.
\]
Moreover, equality is attained by the Koebe function $k(z)=\frac{z}{(1-z)^2},\;\; z\in\mathbb{D},$
and its rotations.

\begin{defi}
Let $f(z)=z+\sum_{n=2}^{\infty}a_nz^n\in\mathcal{A}.$
For integers \(m,n\ge2\), the \emph{generalized Zalcman functional} is defined by
$J_{m,n}(f):=a_ma_n-a_{m+n-1}.$
In particular, $J_{2,3}(f)=a_2a_3-a_4.$
\end{defi}

Ma~\cite{M1999} conjectured that if \(f\in\mathcal{S}\), then
\[
|J_{m,n}(f)|\le (m-1)(n-1), \qquad m,n\ge2.
\]
Furthermore, he proved this conjecture for the class \(\mathcal{S}^{*}\), and also for functions in \(\mathcal{S}\) whose Taylor coefficients are all real.\\

\medskip
In 2018, Cho \emph{et al.}~\cite{CKKLS2018} established the following sharp estimates for the second-order Hankel determinant $H_{2,2}(f)$ and the Fekete--Szeg\"o functional $H_{2,1}(f)$ in the class of starlike functions of order $\alpha$.

\begin{theoA}\cite[Theorem 2.3]{CKKLS2018}
If $f(z)=z+\sum_{n=2}^{\infty}a_nz^n
\in\mathcal{S}^*(\alpha),
\;0\leq \alpha<1,$
then
\[
|a_2a_4-a_3^2|\le (1-\alpha)^2,
\]
and the estimate is sharp.
\end{theoA}

\begin{theoB}\cite[Theorem 2.5]{CKKLS2018}
If $f(z)=z+\sum_{n=2}^{\infty}a_nz^n
\in\mathcal{S}^*(\alpha),
\;0\leq \alpha<1,$
then
\[
|a_3-a_2^2|\le 1-\alpha,
\]
and the estimate is sharp.
\end{theoB}
In the same paper, Cho \emph{et al.}~\cite{CKKLS2018} established the following sharp estimates for the Zalcman functional $J_{2,3}(f)$ in the class of starlike functions of order $\alpha$.
\begin{theoC}\cite[Theorem 2.1]{CKKLS2018}
If $f(z)=z+\sum_{n=2}^{\infty}a_nz^n
\in\mathcal{S}^*(\alpha),
\;0\leq \alpha<1,$
then
\[
|a_2a_3-a_4|\le \begin{cases} \frac{2}{3}(1-\alpha)[4(1-\alpha)^2-1], \;\text{for}\; & 0\leq \alpha\leq  1-\frac{\sqrt 3}{2}\\
         \frac{2(1-\alpha)}{3\sqrt{1-(1-\alpha)^2}} \;\text{for}\; &  1-\frac{\sqrt 3}{2}<\alpha <1
\end{cases}
\]
and the estimate is sharp.
\end{theoC}
Many authors have investigated various subclasses of biholomorphic mappings in Banach spaces and their associated coefficient problems. Significant contributions in this direction can be found in \cite{HKK2021,HKL2001,H2023,Kohr1998,Suffridge1973,XDL2026,XHX2025}.

\medskip

The sharp estimates obtained by Cho \emph{et al.} \cite{CKKLS2018} naturally raise the following question: 
\begin{ques} Do analogous sharp estimates hold for starlike mappings of order $\alpha$ in complex Banach spaces? 
\end{ques} 
The main purpose of this paper is to answer this question affirmatively. Specifically, we establish sharp estimates for the second-order Hankel determinant, the Fekete--Szeg\"o functional and the Zalcman functional associated with the class of starlike mappings of order $\alpha$ on the unit ball of a complex Banach space. Our results extend the corresponding classical one-variable results to the infinite-dimensional setting.

The remainder of this paper is organized as follows. Section~2 is devoted to preliminary results, where we collect several useful known lemmas and establish some new auxiliary lemmas. In Section~3, we investigate the second-order Hankel determinant and the Fekete--Szeg\"o functional, obtaining sharp estimates and identifying the corresponding extremal mappings. In Section~4, we study the Zalcman functional and derive sharp bounds together with the corresponding extremal results.
\section{{\bf Lemmas}}
Before proving the sharpness of the main theorem, we establish the following auxiliary lemmas. These results will be used repeatedly in the subsequent analysis and form the basis for the proof of the sharpness of the obtained estimate.
\begin{lem}\label{L1}\cite[Lemma 2.1]{XHX2025}
Let $f\in S$, $0<\alpha<1$, $x_0\in X$ with $\|x_0\|=1$, and let
$T_{x_0}\in T(x_0)$. Define by $\dfrac{f\!\left(T_{x_0}(x)\right)}{T_{x_0}(x)}\,x$ belongs to $\mathcal{S}_{\mathbb{B}}^{*}(\alpha)$ if and only if $f\in \mathcal{S}^{*}(\alpha)$.
In particular, $F\in \mathcal{S}_{\mathbb{B}}^{*}$ if and only if 
$f\in \mathcal{S}^{*}.$
\end{lem}
\begin{lem}\label{L2}\cite[Lemma 2.2]{XHX2025}
Suppose that $g\in H(\mathbb{B},\mathbb{C})$, $g(0)=1$, and $F(x)=g(x)x,\; x\in \mathbb{B},$
where $0<\alpha<1$. Fix $x_0\in\partial \mathbb{B}$ and define $f(\xi)=g(\xi x_0)\xi,\; \xi\in\mathbb{U}.$
Then
\[
f\in \mathcal{S}^{*}(\alpha)
\quad\Longleftrightarrow\quad
F\in \mathcal{S}_\mathbb{B}^{*}(\alpha),
\]
and
\[
f\in \mathcal{S}^{*}
\quad\Longleftrightarrow\quad
F\in \mathcal{S}_\mathbb{B}^{*}.
\]
\end{lem}

\begin{lem}\label{L3}\cite[Theorem 2.1]{CKKLS2018}
Suppose that $f\in \mathcal{S}^{*}(\alpha)$ and is given by \eqref{eq:A}. Then
\begin{equation}\label{eq:coefficients}
\begin{aligned}
a_2 &=(1-\alpha)c_1,\\[2mm]
a_3 &= \frac{1}{2}(1-\alpha)\left(c_2+(1-\alpha)c_1^2\right),\\[2mm]
a_4 &= \frac{1}{6}(1-\alpha)\left(2c_3+3(1-\alpha)c_1c_2+(1-\alpha)^2c_1^3\right),
\end{aligned}
\end{equation}
where $c_1$, $c_2$, and $c_3$ are given by \eqref{eq:P}.
\end{lem}
\begin{lem}\label{L4}  Let $g\in H(\mathbb{B},\mathbb{C})$ satisfy $g(0)=1$, and define $F(x)=g(x)x, x\in\mathbb{B}.$ Suppose that $F\in\mathcal{S}_{\mathbb{B}}^{*}(\alpha)$. Then, for every
$x_{0}\in X$ with $\|x_{0}\|=1$,
\bea\label{A234}
\begin{cases} A_{2}
&=
\frac{1}{2!}
T_{x_{0}}
\left(
D^{2}F(0)(x_{0}^{2})
\right)
=(1-\alpha)c_1,\\
A_{3}
&=
\frac{1}{3!}
T_{x_{0}}
\left(
D^{3}F(0)(x_{0}^{3})
\right)
=\frac{1}{2}(1-\alpha)\left(c_2+(1-\alpha)c_1^2\right),
\\
\text{and}\;\; 
A_{4}
&=
\frac{1}{4!}T_{x_{0}}\left(D^{4}F(0)(x_{0}^{4})\right)=\frac{1}{6}(1-\alpha)\left(2c_3+3(1-\alpha)c_1c_2+(1-\alpha)^2c_1^3\right),
\end{cases}
\eea
where $c_1$, $c_2$ and $c_3$ are given by (\ref{eq:P}).

\end{lem}
\begin{proof}
Fix $x_{0}\in\partial\mathbb{B}$ and define $f(\xi)=g(\xi x_{0})\,\xi, \xi\in\mathbb{U}.$
Since $F\in\mathcal{S}_{\mathbb{B}}^{*}(\alpha)$, Lemma~\ref{L2}
implies that $f\in \mathcal{S}^*(\alpha).$

On the other hand,
\[
f(\xi)
=
T_{x_{0}}\bigl(F(\xi x_{0})\bigr),
\]
and hence
\[
a_{2}
=
\frac{f''(0)}{2!}
=
\frac{1}{2!}
T_{x_{0}}
\left(
D^{2}F(0)(x_{0}^{2})
\right)
=
A_{2},
\]
\[
a_{3}
=
\frac{f^{(3)}(0)}{3!}
=
\frac{1}{3!}
T_{x_{0}}
\left(
D^{3}F(0)(x_{0}^{3})
\right)
=
A_{3},
\]
and \[
a_{4}
=
\frac{f^{(4)}(0)}{4!}
=
\frac{1}{4!}
T_{x_{0}}
\left(
D^{4}F(0)(x_{0}^{4})
\right)
=
A_{4}.
\]
Applying Lemma~\ref{L3}, we obtain the desired result.
\end{proof}

\begin{lem}\label{L5} Let $F$ be a locally biholomorphic mapping on $\mathbb{B}$ and $F$ satisfy the following assumptions:
\begin{equation}\label{cc1}
\frac{D^{k+1}F(0)\left(x^{k+1}\right)}{(k+1)!}
= H_{F,k}(x)x,\qquad x\in X,\quad k=1,2,3.
\end{equation}
where $H_{F,k}(x)$ is a homogeneous polynomial of degree $k$ with values in $\mathbb{C}$. If $F\in \mathcal{S}_{\mathbb{B}}^{*}(\alpha)$, then for each $x_0\in X$ with $\|x_0\|=1$,  the quantities $A_2$, $A_3$ and $A_4$ are given by  (\ref{A234}).
\end{lem}
\begin{proof} Fix $x_0\in \partial B$ and let $T_{x_0}\in T(x_0)$. Define
\bea\label{aaa1}
\phi(\xi)=
\begin{cases}
\dfrac{\xi}{T_{x_0}\!\left(\varphi(\xi x_0)\right)}, & \xi\neq 0,\\[2mm]
1, & \xi=0,
\end{cases}
\eea
where
\[
\varphi(x)=(DF(x))^{-1}F(x).
\]
Then $\phi\in H(\mathbb U)$ and $\phi(0)=1$. Moreover, since
$F\in \mathcal{S}_{\mathbb{B}}^{*}(\alpha)$ , it follows from the definition of Ozaki close to convex function that
\[
\Re\!\left(\frac{\xi}{T_{x_0}\!\left(\varphi(\xi x_0)\right)}\right)
=
\Re\!\left(\frac{\|\xi x_0\|}{T_{\xi x_0}\!\left(\varphi(\xi x_0)\right)}\right)>\alpha,\;\;\text{for}\;\;0\leq \alpha < 1
\qquad \xi\in\mathbb U\setminus\{0\}.
\]
Hence,
\[\Re\!\left(\phi(\xi)\right)>\alpha,\;\;\text{for}\;\;0\leq \alpha < 1\qquad \xi\in\mathbb U.
\]
 Then there exists $p\in \mathcal{P}$ such that 
\bea\label{kp1}\frac{1}{1-\alpha}(\phi(\xi)-\alpha)=p(\xi).\eea
Since $\phi$ is holomorphic function on $\mathbb{U}$, then it has Taylor series expansion:
\bea\label{phi1}\phi(\xi)=1+\frac{T_{x_0}(D^2\varphi(0)(x_0^2))}{2!}\xi+\frac{T_{x_0}(D^3\varphi(0)(x_0^3))}{3!}\xi^2+\frac{T_{x_0}(D^4\varphi(0)(x_0^4))}{4!}\xi^3+\cdots,\xi\in \mathbb{U}.\eea
Using (\ref{aaa1}), \eqref{eq:P}, \eqref{kp1} and \eqref{phi1}, a direct computation yields
\bea\label{qq1}
\begin{cases} \frac{T_{x_0}(D^2\varphi(0)(x_0^2))}{2!}&=-(1-\alpha)c_1,\\
 \frac{T_{x_0}(D^3\varphi(0)(x_0^3))}{3!}&=(1-\alpha)^2 c_1^2-(1-\alpha)c_2\\
\frac{T_{x_0}(D^4\varphi(0)(x_0^4))}{4!}&=-(1-\alpha)^3 c_1^3+2(1-\alpha)^2 c_1c_2-(1-\alpha)c_3.
\end{cases}
\eea
On the other hand, since $\varphi(x)=(DF(x))^{-1}F(x)$, it follows that
\bea\label{qqq2}
\begin{cases}
\frac{D^{2}F(0)(x^{2})}{2!}&=-\frac{D^{2}\varphi(0)(x^{2})}{2!}\\
 -\frac{D^{3}F(0)(x^{3})}{3}&=\frac{D^{3}\varphi(0)(x^{3})}{3!}+D^{2}F(0)\!\left(x,\,\frac{D^{2}\varphi(0)(x^{2})}{2!}\right)\\
-\frac{D^{4}F(0)(x^{4})}{8}&=\frac{D^{4}\varphi(0)(x^{4})}{4!}+D^{2}F(0)\!\left(x,\,\frac{D^{3}\varphi(0)(x^{3})}{3!}\right)\\
&\;\;\;+\frac{1}{2}D^{3}F(0)\!\left(x^{2},\,\frac{D^{2}\varphi(0)(x^{2})}{2!}\right).
\end{cases}\eea

Combining(\ref{qq1}) and (\ref{qqq2}) together with assumption \eqref{cc1}, we derive explicit formulas for the coefficients \(A_2\), \(A_3\)  and \(A_4\) in terms of the parameters \(p_1\), \(p_2\) and  \(p_3\), respectively.

\bea\label{A2}A_2&=&T_{x_0}\!\left(\frac{D^2F(0)(x_0^2)}{2!}\right)=H_{F,1}(x_0)
=-\,T_{x_0}\!\left(\frac{D^2\varphi(0)(x_0^2)}{2!}\right)
=(1-\alpha)c_1\eea
and
\[
\frac{D^2\varphi(0)(x_0^2)}{2!}=-\,\frac{D^2F(0)(x_0^2)}{2!}=-H_{F,1}(x_0)x_0
=-A_2x_0
=-(1-\alpha)c_1x_0.
\]

\bea\label{A3} A_3&=&T_{x_0}\!\left(\frac{D^3F(0)(x_0^3)}{3!}\right)
=H_{F,2}(x_0) \nonumber\\
&=&-\frac{1}{2}\left(T_{x_0}\!\left(\frac{D^3\varphi(0)(x_0^3)}{3!}\right)
+T_{x_0}\!\left(D^2F(0)\!\left(x_0,\frac{D^2\varphi(0)(x_0^2)}{2!}\right)\right)\right) \nonumber\\
&=&-\frac{1}{2}\left((1-\alpha)^2 c_1^2-(1-\alpha)c_2+T_{x_0}\!\left(D^2F(0)(x_0,-(1-\alpha)c_1x_0\right)\right) \nonumber\\
&=&-\frac{1}{2}\left((1-\alpha)^2 c_1^2-(1-\alpha)c_2-2(1-\alpha)c_1A_2\right)\nonumber\\
&=&\frac{1-\alpha}{2}\left(c_2+(1-\alpha) c_1^2\right)
\eea

and

\beas 
\frac{D^3\varphi(0)(x_0^3)}{3!}&=&-\frac{D^3F(0)(x_0^3)}{3}-D^2F(0)\!\left(x_0,\frac{D^2\varphi(0)(x_0^2)}{2!}\right)\\
&=&-\frac{D^3F(0)(x_0^3)}{3}-D^2F(0)(x_0,(1-\alpha) c_1 x_0)\\
&=&\left(-2H_{F,2}(x_0)-2(1-\alpha)^2 b_1^2\right) x_0\\
&=&(1-\alpha)((1-\alpha)c_1^2-c_2)x_0.\eeas
\bea\label{A4}
A_4&=& T_{x_0}\!\left(\frac{D^4F(0)(x_0^4)}{4!}\right)=H_{F,3}(x_0) \nonumber\\
&=&-\frac{1}{3}\Bigg(T_{x_0}\!\left(\frac{D^4\varphi(0)(x_0^4)}{4!}\right)+T_{x_0}\!\left(D^2F(0)\!\left(x_0,\frac{D^3\varphi(0)(x_0^3)}{3!}\right)\right) \nonumber\\
&& +\frac{1}{2}T_{x_0}\!\left(D^3F(0)\!\left(x_0^2,\frac{D^2\varphi(0)(x_0^2)}{2!}\right)\right)\Bigg) \nonumber\\
&=&-\frac{1}{3}\Bigg(-(1-\alpha)^3 c_1^3+2(1-\alpha)^2 c_1c_2-(1-\alpha)c_3 \nonumber\\
&&+
T_{x_0}\!\left(D^2F(0)(x_0,(1-\alpha)((1-\alpha)c_1^2-c_2)x_0)\right)+\frac{1}{2}T_{x_0}\!\left(D^3F(0)(x_0^2,-(1-\alpha)c_1x_0)\right)\Bigg) \nonumber\\
&=&-\frac{1}{3}\Bigg(-(1-\alpha)^3 c_1^3+2(1-\alpha)^2 c_1c_2-(1-\alpha)c_3\nonumber\\
&&+2(1-\alpha)((1-\alpha)c_1^2-c_2)x_0A_2-3(1-\alpha)c_1A_3\Big) \nonumber\\
&=&\frac{1}{6}(1-\alpha)\left((1-\alpha)^2 c_1 +3(1-\alpha)c_1c_2+2c_3\right).\eea
\end{proof}

\section{\texorpdfstring{\bf The Second-Order Hankel Determinant for the Class $\mathcal{S}^*_\mathbb{B}(\alpha)$}{The Second-Order Hankel Determinant}}

In this section, we study the second-order Hankel determinant
\[
H_{2,2}(F)=\begin{vmatrix}
A_2 &A_3\\
 A_3 & A_4 
\end{vmatrix}=A_2A_4-A_3^2\;\text{and}\; H_{2,1}(F)=\begin{vmatrix}
1 &A_2\\
 A_2 & A_3 
\end{vmatrix}=A_3-A_2^2,
\] where $A_2$, $A_3$ and $A_4$ are defined by
\eqref{eq:An} for the class $\mathcal{S}^*_\mathbb{B}(\alpha)$. We derive sharp upper bounds for $H_{2,2}(F)$, $H_{2,1}(F)$  and characterize the extremal mappings for which equality are attained, thereby establishing the sharpness of the obtained results.
\begin{theo}\label{T1}
Let $g\in H(\mathbb{B},\mathbb{C})$ satisfy $g(0)=1$, and define $F(x)=g(x)x, x\in\mathbb{B}.$
Suppose that $F\in\mathcal{S}^*_\mathbb{B}(\alpha)$, $0\leq \alpha <1$. Then, for every
$x_{0}\in X$ with $\|x_{0}\|=1$,
\[
|H_{2,2}(F)|\leq (1-\alpha)^2,
\]
 Moreover, the estimate is sharp.
\end{theo}

\begin{proof}
 Fix $x_0\in \partial{\mathbb{B}}$ and let $T_{x_0}\in T(x_0)$. Then applying Lemma \ref{L4} we deduce that

\[
|H_{2,2}(F)|
=
|H_{2,2}(f)|.
\]
Applying Theorem~A (see \cite[Theorem 2.3]{CKKLS2018}),
we obtain
\[
|H_{2,2}(F)| \leq (1-\alpha)^2.\]

This completes the proof.
\end{proof}

The following example demonstrates that the estimate obtained in Theorem \ref{T1} is sharp.

\begin{exm}
Consider the mapping
\bea\label{ex1}
F(x)=
\frac{1}{\left(1-(T_{x_0}(x))^2\right)^{1-\alpha}}\,x,
\qquad
x_0\in\partial\mathbb{B},\;
T_{x_0}\in T(x_0).
\eea
By Lemma \ref{L1}, we have $F\in\mathcal{S}^*_\mathbb{B}(\alpha).$
Furthermore, a straightforward computation gives
\[
\frac{D^2F(0)(x^2)}{2!}=0,\;\;\;
\frac{D^3F(0)(x^3)}{3!}
=
(1-\alpha)
\bigl(T_{x_0}(x)\bigr)^2x,
\;\;\;\text{and}\;\;\;
\frac{D^4F(0)(x^4)}{4!}=0.
\]

Hence,
\bea\label{Q1}
A_2=0,\;\;
A_3=1-\alpha,\;\;\text{and}\;\;
A_4=0.
\eea
Therefore,
\[
|H_{2,2}(F)|
=
\left|A_2A_4-A_3^2\right|
=
(1-\alpha)^2.
\]

Thus, the upper bound obtained in Theorem \ref{T1} is attained, showing that the estimate is sharp.
\end{exm}

\medskip
Next, removing the restrictive assumption $F(x)=g(x)x$, we generalize Theorem~A to higher dimensions under weaker assumptions than those of Theorem \ref{T1}. Assume that
\begin{equation}\label{cc1}
\frac{D^{k+1}F(0)\left(x^{k+1}\right)}{(k+1)!}
= H_{F,k}(x)x,\qquad x\in X,\quad k=1,2,3.
\end{equation}
where $H_{F,k}(x)$ is a homogeneous polynomial of degree $k$ with values in $\mathbb{C}$. The fact that the assumption \eqref{cc1} is weaker than that of Theorem \ref{T1}  is justified in \cite{H2023}.

\begin{theo}\label{T1.1}
Let $F$ be a locally biholomorphic mapping on $\mathbb{B}$, and suppose that $F$ satisfies the assumption \eqref{cc1}. If $F\in\mathcal{S}^*_\mathbb{B}(\alpha)$, $0\leq \alpha <1$, then for every $x_0\in X$ with $\|x_0\|=1$, we have
\[
|H_{2,2}(F)|=|A_2A_4-A_3^2|\leq (1-\alpha)^2,
\]
 where $A_1=1$, and $A_2$, $A_3$, and $A_4$ are defined by
\eqref{eq:An}, respectively. Moreover, the above estimate is sharp.
\end{theo}
\begin{proof} Fix $x_0\in \partial{\mathbb{B}}$ and let $T_{x_0}\in T(x_0)$. Next, applying Lemma \ref{L5} and proceeding analogously to the proofs of Theorems A  (see \cite[Theorem 2.3]{CKKLS2018}), we deduce that
\begin{align*}
|T_{2,2}(F)|\leq (1-\alpha)^2.
\end{align*}
 The sharpness of the result is an immediate consequence of Theorem \ref{T1}. This completes the proof.\end{proof}
 \begin{rem}
An immediate consequence of Theorems~\ref{T1} and~\ref{T1.1} is obtained by setting $\alpha=0$. In this case, we have the sharp estimate
$|H_{2,2}(F)|\leq 1$
for the class $\mathcal{S}_{\mathbb{B}}^{*}$ of starlike mappings on the unit ball of a complex Banach space.

Furthermore, if $\mathbb{B}=\mathbb{U}$ and $X=\mathbb{C}$, then $\mathcal{S}_{\mathbb{B}}^{*}$ coincides with the classical class $\mathcal{S}^{*}$ of starlike univalent functions on the unit disk. Consequently, the above estimate reduces to $|H_{2,2}(f)|\leq 1,$
which agrees with the sharp bound established by Janteng \emph{et al.}~\cite[Theorem~3.1]{JHD2007}.

Moreover, under the specialization $\mathbb{B}=\mathbb{U}$ and $X=\mathbb{C}$, Theorems~\ref{T1} and~\ref{T1.1} recover Theorem~A of Cho \emph{et al.}~\cite[Theorem~2.5]{CKKLS2018}. Thus, our results provide a genuine extension of the corresponding one-dimensional theorem to the setting of complex Banach spaces.
\end{rem} 
\begin{theo}\label{T2}
Let $g\in H(\mathbb{B},\mathbb{C})$ satisfy $g(0)=1$, and define $F(x)=g(x)x, x\in\mathbb{B}.$
Suppose that $F\in\mathcal{S}^*_\mathbb{B}(\alpha)$, $0\leq \alpha <1$. Then, for every
$x_{0}\in X$ with $\|x_{0}\|=1$,
\[
|H_{2,1}(F)|\leq 1-\alpha,
\]
 Moreover, the estimate is sharp.
\end{theo}

\begin{proof}
 Fix $x_0\in \partial{\mathbb{B}}$ and let $T_{x_0}\in T(x_0)$. Then applying Lemma \ref{L4} we deduce that

\[
|H_{2,1}(F)|
=
|H_{2,1}(f)|.
\]
Applying Theorem~B (see \cite[Theorem 2.3]{CKKLS2018}),
we obtain
\[
|H_{2,1}(F)| \leq 1-\alpha.\]

This completes the proof.
\end{proof}

The following example demonstrates that the estimate obtained in Theorem \ref{T1} is sharp.

\begin{exm}
Consider the mapping $F$ define in (\ref{ex1}). 
Therefore from (\ref{Q1})) we have 
\[
|H_{2,1}(F)|
=
\left|A_3-A_2^2\right|
=(1-\alpha).
\]

Thus, the upper bound obtained in Theorem \ref{T2} is attained, showing that the estimate is sharp.
\end{exm}

\medskip
Next, removing the restrictive assumption $F(x)=g(x)x$, we generalize Theorem~B to higher dimensions under weaker assumptions \eqref{cc1} than those of Theorem \ref{T2}. 
\begin{theo}\label{T2.1}
Let $F$ be a locally biholomorphic mapping on $\mathbb{B}$, and suppose that $F$ satisfies the assumption \eqref{cc1}. If $F\in\mathcal{S}^*_\mathbb{B}(\alpha)$, $0\leq \alpha <1$, then for every $x_0\in X$ with $\|x_0\|=1$, we have
\[
|H_{2,1}(F)|=|A_3-A_2^2|\leq 1-\alpha,
\]
 where $A_1=1$, and $A_2$ and $A_3$ are defined by
\eqref{eq:An}, respectively. Moreover, the above estimate is sharp.
\end{theo}
\begin{proof} Fix $x_0\in \partial{\mathbb{B}}$ and let $T_{x_0}\in T(x_0)$. Next, applying Lemma \ref{L5} and proceeding analogously to the proofs of Theorems B (see \cite[Theorem 2.5]{CKKLS2018}), we deduce that
\begin{align*}
|T_{2,1}(F)|\leq 1-\alpha.
\end{align*}
 The sharpness of the result is an immediate consequence of Theorem \ref{T2}. This completes the proof.\end{proof}
 \begin{rem}
Setting $\alpha=0$ in Theorems~\ref{T2} and~\ref{T2.1}, we obtain the sharp estimate
\[
|H_{2,1}(F)|\leq 1
\]
for the class $\mathcal{S}_{\mathbb{B}}^{*}$ of starlike mappings on the unit ball of a complex Banach space.

In particular, when $\mathbb{B}=\mathbb{U}$ and $X=\mathbb{C}$, the class $\mathcal{S}_{\mathbb{B}}^{*}$ reduces to the classical class $\mathcal{S}^{*}$ of starlike univalent functions on the unit disk, and the above estimate becomes
\[
|H_{2,1}(f)|\leq 1,
\]
which coincides with the sharp Fekete--Szeg\"o inequality for the class $\mathcal{S}^{*}$.

Furthermore, under the specialization $\mathbb{B}=\mathbb{U}$ and $X=\mathbb{C}$, Theorems~\ref{T2} and~\ref{T2.1} reduce precisely to Theorem~B of Cho \emph{et al.}~\cite[Theorem~2.3]{CKKLS2018}. Consequently, our results extend the corresponding one-dimensional theorem to the framework of complex Banach spaces.
\end{rem}
 
 \section{\texorpdfstring{\bf The Zalcman Functional $J_{2,3}(F)$ for the Class $\mathcal{S}^*_\mathbb{B}(\alpha)$}{The Zalcman Functional}}

In this section, we investigate the Zalcman functional $J_{2,3}(F)$ for the class $\mathcal{S}^*_\mathbb{B}(\alpha)$. We derive sharp upper bounds for this functional and identify the corresponding extremal mappings, thereby establishing the sharpness of the obtained results.
\begin{theo}\label{T3}
Let $g\in H(\mathbb{B},\mathbb{C})$ satisfy $g(0)=1$, and define $F(x)=g(x)x, x\in\mathbb{B}.$
Suppose that $F\in\mathcal{S}^*_\mathbb{B}(\alpha)$, $0\leq \alpha <1$. Then, for every
$x_{0}\in X$ with $\|x_{0}\|=1$,
\[
|J_{2,3}(F)|\leq \begin{cases} \frac{2}{3}(1-\alpha)[4(1-\alpha)^2-1], \;\text{for}\; & 0\leq \alpha\leq  1-\frac{\sqrt 3}{2}\\
         \frac{2(1-\alpha)}{3\sqrt{1-(1-\alpha)^2}} \;\text{for}\; &  1-\frac{\sqrt 3}{2}<\alpha <1
\end{cases},
\]
where $J_{2,3}(F)=A_2A_3-A_4$,  $A_{2},A_{3},A_{4}$ defined by
\eqref{eq:An}. Moreover, the estimate is sharp.
\end{theo}

\begin{proof}
 Fix $x_0\in \partial{\mathbb{B}}$ and let $T_{x_0}\in T(x_0)$. Then applying Lemma \ref{L4} we deduce that

\[
|J_{2,3}(F)|
=
|J_{2,3}(f)|.
\]
Applying Theorem~C (see \cite[Theorem 2.1]{CKKLS2018}),
we obtain
\[
|J_{2,3}(F)| \leq \begin{cases} \frac{2}{3}(1-\alpha)[4(1-\alpha)^2-1], \;\text{for}\; & 0\leq \alpha\leq  1-\frac{\sqrt 3}{2}\\
         \frac{2(1-\alpha)}{3\sqrt{1-(1-\alpha)^2}} \;\text{for}\; &  1-\frac{\sqrt 3}{2}<\alpha <1
\end{cases}.\]

This completes the proof.
\end{proof}

The following example demonstrates that the estimate obtained in Theorem \ref{T3} is sharp.

\begin{exm} Suppose that $0\leq \alpha\leq  1-\frac{\sqrt 3}{2}$.
Consider the mapping
\bea\label{ex3}
F(x)=
\frac{1}{\left(1-T_{x_0}(x)\right)^{2(1-\alpha)}}\,x,
\qquad
x_0\in\partial\mathbb{B},\;
T_{x_0}\in T(x_0).
\eea
By Lemma \ref{L1}, we have $F\in\mathcal{S}^*_\mathbb{B}(\alpha).$
A simple calculation shows that
\beas
\frac{D^2F(0)(x^2)}{2!}&=&2(1-\alpha)\bigl(T_{x_0}(x)\bigr)x,\\
\frac{D^3F(0)(x^3)}{3!}
&=&(1-\alpha)(3-2\alpha)
\bigl(T_{x_0}(x)\bigr)^2x\\
\text{and}\;\;\;
\frac{D^4F(0)(x^4)}{4!}&=&\frac{2(1-\alpha)(3-2\alpha)(2-\alpha)}{3}
\bigl(T_{x_0}(x)\bigr)^3x.
\eeas

Hence,
\[
A_2=2(1-\alpha),\;\;
A_3=(1-\alpha)(3-2\alpha),\;\;\text{and}\;\;
A_4=\frac{2(1-\alpha)(3-2\alpha)(2-\alpha)}{3}.
\]
Therefore,
\[
|J_{2,3}(F)|
=
\left|A_2A_3-A_4\right|
=\frac{2}{3}(1-\alpha)[4(1-\alpha)^2-1].
\]

Next suppose that $ 1-\frac{\sqrt 3}{2}<\alpha<1$.
Consider the mapping
\bea\label{ex3}
F(x)=
\frac{1}{\left(1-aT_{x_0}(x)+(T_{x_0}(x))^2 \right)^{1-\alpha}}\,x,
\qquad
x_0\in\partial\mathbb{B},\;
T_{x_0}\in T(x_0),
\eea
where $a=\frac{1}{\sqrt{1-(1-\alpha)^2}}$.
By Lemma \ref{L1}, we have $F\in\mathcal{S}^*_\mathbb{B}(\alpha).$
Computing directly, we obtain
\beas
\frac{D^2F(0)(x^2)}{2!}&=&\frac{1-\alpha}{\sqrt{\alpha(2-\alpha)}}\bigl(T_{x_0}(x)\bigr)x,\\
\frac{D^3F(0)(x^3)}{3!}
&=&\frac{(1-\alpha)(1-2\alpha)}{2\alpha}
\bigl(T_{x_0}(x)\bigr)^2x\\
\text{and}\;\;\;
\frac{D^4F(0)(x^4)}{4!}&=&\frac{(1-\alpha)(2-\alpha)(3-13\alpha+6\alpha^2)}{6\bigl(\alpha(2-\alpha)\bigr)^{3/2}}\bigl(T_{x_0}(x)\bigr)^3x.
\eeas

Hence,
\[
A_2=\frac{1-\alpha}{\sqrt{\alpha(2-\alpha)}},\;\;
A_3=\frac{(1-\alpha)(1-2\alpha)}{2\alpha},\;\;\text{and}\;\;
A_4=\frac{(1-\alpha)(2-\alpha)(3-13\alpha+6\alpha^2)}{6\bigl(\alpha(2-\alpha)\bigr)^{3/2}}.
\]
Therefore,
\[
|J_{2,3}(F)|
=
\left|A_2A_3-A_4\right|
=\frac{2(1-\alpha)}{3\sqrt{1-(1-\alpha)^2}}.
\]

Thus, the upper bound obtained in Theorem \ref{T3} is attained, showing that the estimate is sharp.
\end{exm}

\medskip
Next, by removing the restrictive assumption $F(x)=g(x)x$, we generalize Theorem~C to higher dimensions under the weaker assumption \eqref{cc1}, thereby relaxing the assumptions of Theorem~\ref{T3}.
\begin{theo}\label{T3.1}
Let $F$ be a locally biholomorphic mapping on $\mathbb{B}$, and suppose that $F$ satisfies the assumption \eqref{cc1}. If $F\in\mathcal{S}^*_\mathbb{B}(\alpha)$, $0\leq \alpha <1$, then for every $x_0\in X$ with $\|x_0\|=1$, we have
\[
|J_{2,3}(F)| \leq \begin{cases} \frac{2}{3}(1-\alpha)[4(1-\alpha)^2-1], \;\text{for}\; & 0\leq \alpha\leq  1-\frac{\sqrt 3}{2}\\
         \frac{2(1-\alpha)}{3\sqrt{1-(1-\alpha)^2}} \;\text{for}\; &  1-\frac{\sqrt 3}{2}<\alpha <1
\end{cases}.\]
  Moreover, the above estimate is sharp.
\end{theo}
\begin{proof} Fix $x_0\in \partial\mathbb{B}$ and let $T_{x_0}\in T(x_0)$. Applying Lemma~\ref{L5} and arguing as in the proof of Theorem C (cf. \cite[Theorem~2.1]{CKKLS2018}), we obtain
\[
|J_{2,3}(F)| \leq \begin{cases} \frac{2}{3}(1-\alpha)[4(1-\alpha)^2-1], \;\text{for}\; & 0\leq \alpha\leq  1-\frac{\sqrt 3}{2}\\
         \frac{2(1-\alpha)}{3\sqrt{1-(1-\alpha)^2}} \;\text{for}\; &  1-\frac{\sqrt 3}{2}<\alpha <1
\end{cases}.\]
The sharpness follows immediately from Theorem~\ref{T3}. This completes the proof.\end{proof}
\begin{rem}
Setting $\alpha=0$ in Theorems~\ref{T3} and \ref{T3.1}, we obtain the sharp estimate $|J_{2,3}(F)|\leq 2$
 for the class $\mathcal{S}_{\mathbb{B}}^{*}$ of starlike mappings on the unit ball of a complex Banach space.
 
 In particular, when $\mathbb{B}=\mathbb{U}$ and $X=\mathbb{C}$, the above estimate reduces to
$|J_{2,3}(f)|\leq 2,$
which is precisely the classical sharp bound for the Zalcman functional in the class $\mathcal{S}^{*}$ of starlike univalent functions on the unit disk.

Moreover, in the one-dimensional setting $\mathbb{B}=\mathbb{U}$ and $X=\mathbb{C}$, Theorems~\ref{T3} and~\ref{T3.1} reduce precisely to Theorem~C of Cho \emph{et al.}~\cite[Theorem~2.1]{CKKLS2018}.

\end{rem}

\section*{{\bf Declarations}}
\subsection*{Funding}
The second author acknowledge financial support from the Council of Scientific and Industrial Research (CSIR), New Delhi, India, under Grant Nos. 09/1224(16975)/2023-EMR-I.
\subsection*{Data Availability Statement}
Data sharing is not applicable to this article as no datasets were generated or analyzed during the current study.
\subsection*{Conflict of Interest}
The authors declare that they have no conflict of interest. 
\subsection*{Author Contributions}
Both authors contributed equally to this work.

\end{document}